\documentclass[a4paper,11pt,reqno]{amsart}
\usepackage[T1]{fontenc}
\usepackage{amsmath}

\usepackage{tikz}
\usetikzlibrary{calc}
\usepackage{tikz}
\usetikzlibrary{decorations.markings}

\usepackage{latexsym,amsfonts,amssymb,amsthm,euscript, subfigure}
\usepackage{graphicx}
\usepackage{hyperref}
\usepackage{cleveref}
\usepackage{float}
\usepackage{color,fancybox}
\usepackage[normalem]{ulem}  
\usepackage{latexsym}
\usepackage{enumitem}
\usepackage{comment}
\usepackage[british]{babel}

\usepackage{mathrsfs}

\newcounter{hipotese}

\theoremstyle{plain}
\newtheorem{Teo}{Theorem}
\newtheorem{Def}[Teo]{Definition}
\newtheorem{Ex}[Teo]{Example}
\newtheorem{Lema}[Teo]{Lemma}
\newtheorem{Prop}[Teo]{Proposition}

\newtheorem{Cor}[Teo]{Corollary}
\newtheorem{maintheorem}{Theorem}

\newtheorem{mar}[Teo]{Remark}

\newcommand{\eps}{\varepsilon}

\begin{document}

\title{On the expansiveness of invariant measures under pseudogroups}

\author[A. Arbieto]{A. Arbieto}
\address{Alexander Arbieto\\ Departamento de Matem\'atica, Instituto de Matem\'atica, Universidade Federal do Rio de Janeiro, Rio de Janeiro, RJ, Brazil}
\email{arbieto@im.ufrj.br}

\author[L. Segantim]{L. Segantim}
\address{Luana Segantim\\ Departamento de Matem\'atica, Instituto de Matem\'atica, Universidade Federal do Rio de Janeiro, Rio de Janeiro, RJ, Brazil}
\email{luanasegantim@gmail.com}

\author[J. Siqueira]{J. Siqueira}
\address{Jaqueline Siqueira\\ Departamento de Matem\'atica, Instituto de Matem\'atica, Universidade Federal do Rio de Janeiro, Rio de Janeiro, RJ, Brazil}
\email{jaqueline@im.ufrj.br}


\begin{abstract} 

In this paper, we define and study weak expansive and expansive measures for pseudogroups, these two notions appear when analyzing the role of the generating set. We investigate the relations between such properties. We also provide a criterion for a measure to be  weak expansive through the positivity of its entropy, generalizing \cite{AM}. We also show that in some settings equicontinuous pseudogroups have no expansive measures.
\end{abstract}

\maketitle

\section{Introduction}

The study of chaotic properties on dynamical systems is a very classical and fruitful area of research. By its nature, it is present in many complicated systems. Moreover, it was detected in hyperbolic dynamics in the celebrated work of Smale \cite{Smale} and also in the classical Lorenz equations, in the study of wheater prediction.

This complicated behavior appears in the literature in several ways, like sensitivity to initial conditions \cite{c.}, positive entropy \cite{pe} or expansivity (as in Smale's work). We are interested in the latter one. 

Expansivity means that there is a certain positive scale, such that if two orbits are close up to this scale at any time then the orbits are actually equal. In other words, different orbits must ``separate''. In the topological fashion this was defined by Utz \cite{MU} for homeomorphisms, and it was discussed for other type of dynamical systems by several authors: for non-singular flows \cite{PW}, for singular flows \cite{Komuro}, for group actions \cite{GK}. In this article, we will focus on pseudogroups. There are two reasons. The first one is that they resemble the theory of group actions but in a very local fashion (which creates several complications). The second one is that they appear naturally as a tool to study the dynamics of foliations, since they can be studied using the pseudogroup of holonomy generated by the foliation.

In recent years, the study of expansive-type properties arise in several works. Notice that the expansivity can be rephrased as the following: take a point $x$ and define as the Bowen ball of radius $c>0$ of $x$ as the set of other points whose orbit keeps $c$-close to the orbit of $x$, then the system is expansive if those Bowen balls reduces to a singleton. So new expansive-type properties deal with the study of such Bowen balls.

Indeed, the notion of \emph{entropy expansive} systems was introduced by Bowen in \cite{BO}, where it is asked that for some $c>0$ the entropy of the Bowen balls is zero. It turns out that this has many implications in ergodic theory, since this implies the upper semicontinuity of the entropy map and, therefore, the existence of equilibrium states. Obviously, this is a weaker notion than expansivity and, later on, it was proved that entropy-expansivity is present for many partially hyperbolic diffeomorphisms and flows  \cite{DI}. It was also used to obtain SRB measures in some contexts, as in \cite{COW}.

Among others expansive-type we are interested in the following ones. In a measure theoretical way, the notion of \emph{expansive measures} appears, where it is asked that the measure of Bowen balls are zero, see \cite{AM}. This notion was studied by several authors and has many consequences. Indeed, in \cite{AM}     it was proved that, for continuous maps, every ergodic measure with positive entropy is expansive, thus recasting again a chaotic behavior. Also, it was considered in     \cite{SIR} the case where the map is \emph{measure expansive}, i.e. when every non-atomic measure is expansive. Actually, it was proved by             \cite{artig} that this is equivalent to the map be \emph{countable expansive}, i.e. the Bowen balls are countable sets.

In this paper we introduce the notion of a expansive measure for pseudogroups, inspired by the notion of expansive pseudogroups \cite{Good}. However, due to the complicated behavior of pseudogroups under compositions, we refer the reader to section    $6.2$ for the precise definition. We then discuss its sensitivity to the choice of generators. In particular, even that for maps the definition of an expansive measure independs if we consider all Bowen balls or almost every Bowen ball, due to the nature of pseudogroups, we need to distinguished such behaviors, then we define expansive measures and weakly expansive measures. Then, we prove the following result, which appears with all of its technical details as Theorem \ref{qtp} in section \ref{s.exp}.

\begin{maintheorem}
Let          $\mathcal{G}$ a good pseudogroup, $\mathcal{G}_{1}$ a good generating set and $\mathcal{G}_{2}$  the compacted generating set. Then, there exists $\rho> 0$ such that  if a measure          $\mu$ is $(\mathcal{G}, \mathcal{G}_{2})$-weakly expansive with constant                    $\rho > 0$     then $\mu$ is $(\mathcal{G}, \mathcal{G}_{1})$-expansive with constant $\frac{\rho}{2}$.
\end{maintheorem}

We also analyze some of its dynamical properties, as the invariance by isomorphisms (Theorem \ref{iso}), and the relations with countable-expansivity (Proposition \ref{cou}). We then extend the result by Arbieto and Morales \cite{AM} for homogeneous measures.

\begin{maintheorem}
 [Criterion for Weakly Expansive Measures]
       Every ergodic invariant and  $(\mathcal{G}, \mathcal{G}_{1})  -$homogeneous measure of a pseudogroup $( \mathcal{G},  \mathcal{G}_{1})$ with  positive  local upper measure entropy is $(\mathcal{G}, \mathcal{G}_{1})$-weakly expansive.
\end{maintheorem}

We remark that as a corollary of the methods, we obtain the same result for arbitrary measures in the case of countable amenable group actions, see Remark   \ref{remark de grupo} and Corollary                                          \ref{cor de gr}.

In the other hand, we analyze equicontinuous pseudogroups, since they are related with Riemannian foliations as in               \cite{Lo} and also with pseudogroups of isometries. It was proved by Dominik Kwietniak and Martha Łącka   in    \cite{Ac} that if $\mathcal{G}$ is pseudogroup acting in a compact metric space $X$ and is generated by a finite a symmetric set satisfying the  uniform equicontinuity  condition and $\mathcal{U}$ is a finite open cover of $X$ then the complexity of $\mathcal{G}$ with respect to $\mathcal{U}$ is bounded. In particular, this gives a clue that expansive measures should not exist, as in the next result.

\begin{maintheorem}
    Let $G$ an uniformly equicontinuous group of homeomorphisms in $X$ finitely generated by $G_{1}$. Then, the pseudogroup $\mathcal{G}(G)$ generated by $G$ has no $(\mathcal{G}(G), G_{1})$-weakly expansive measures.
\end{maintheorem}

This article is organized as follows. In section \ref{s.prel}, we give the necessary definitions that will be used along the text. In section \ref{s.exp}, we define and study expansive measures. In section \ref{s.iso}, we recall the notions of isomorphic pseudogroups and study its relations with expansive measures and entropy. In section \ref{s.arbmor}, we study stable sets and give a criterion to obtain expansive measures through positive entropy. Finally, in section \ref{s.equi}, we show that uniform equicontinuous pseudogroups has no expansive measures. Along the text several questions appears naturally.

\section{Preliminaries}
\label{s.prel}

In this section, we recall the  basic notions needed to give precise statements of the results and prepare some tools for the rest of the article. Even that some definitions holds for general topological spaces, in this article we will always suppose that the phase space $X$ is a compact metric space, and we will consider the Borel $\sigma$-algebra when dealing with probability measures.

\subsection{Pseudogroups}

Let $Homeo(X)$ be the set of all homeomorphisms between open subsets of $X$.          Given $g \in Homeo(X)$ we denote by $D_{g}$ its domain and $R_{g}$ its range. Further, we write $g(A)$ to mean $g(A \cap                   D_{g})$ and              if $h \in Homeo(X)$ then $D_{h \circ g} = g^{-1}(D_{h})$.

We recall    the concept of pseudogroup as presented                                                   in                               \cite{Entropy}.
\begin{Def}
A subset $\mathcal{G} \subset Homeo(X)$ is said to be a pseudogroup if it satisfies the following properties:

\begin{enumerate}[label=(\arabic*)]
    \item If $g, h \in \mathcal{G}$ then $g \circ h \in \mathcal{G}$;
    
    \item If $g \in \mathcal{G}$   then $g^{-1} \in \mathcal{G}$;
    
    \item If $g \in \mathcal{G}$ and $U$ is open subset of $D_{g}$ then $g|_{U} \in \mathcal{G}$;
    
    \item                                  If $g \in  Homeo(X)$, $\mathcal{U}$ is an open cover of $D_{g}$, and $g|_{U} \in \mathcal{G}$, for every $ U \in \mathcal{U}$, then $g \in                              \mathcal{G}$;
    
    \item $id_{X} \in                                                     \mathcal{G}$.
\end{enumerate}
    
\end{Def}


\subsection{Generating Sets}
Now, we will define finitely generated pseudogroups as is presented in \cite{Good}.

\begin{Def}\cite{Good}
    Let       $\Gamma \subset Homeo(X)$   a subset satisfying
    \begin{equation*}
        \bigcup_{g \in  \Gamma} (D_{g} \cup R_{g}) = X.
    \end{equation*}
    We define the pseudogroup $\mathcal{G}(\Gamma)$ generated by $\Gamma$ as the pseudogroup given by the property: $g \in  \mathcal{G}(\Gamma)$  if, and only if, $g \in Homeo(X)$ and  for any $x \in D_{g}$ there are $g_{1}, \ldots, g_{n} \in \Gamma$, $e_{1}, \ldots, e_{n} \in \{\pm          1\}$ and   a neighborhood $U_{x}$ of $x$, with $U_x \subset D_{g}$ and satisfying $g|_{U_{x}} = g_{1}^{e_{1}} \circ \ldots \circ g_{n}^{e_{n}}|_{U_x}$.

    If $\Gamma$ is finite then $\mathcal{G}(\Gamma)$ is  called a finitely generated pseudogroup.
\end{Def}

\begin{Def}\cite{Good}
    A generating set $\Gamma$ is said to be symmetric  if $id_{X}$ and $ \Gamma^{-1}=\{g^{-1}; \ g \in                  \Gamma\}$ are contained in                $\Gamma$.
\end{Def}

\begin{Ex}\cite{Good}
   Let $\mathcal{U}$ a nice covering of a foliated manifold $(M, \mathcal{F})$ and for any $U \in                                     \mathcal{U}    $ consider $T_{U}$ the space of the plaques of $\mathcal{F}$ contained in $U$. Define the \emph{complete transversal} for $\mathcal{F}$ by
   \begin{equation*}
       T = \bigsqcup_{U \in \mathcal{U}} T_{U}.
   \end{equation*}
   Now, let $U, V \in \mathcal{U}$ with $U \cap V \neq \emptyset$ and               consider $D_{VU}$ as the open subset of $T_{U}$   of all plaques $P$ of $U$ such that $P \cap V \neq                                                              \emptyset.$  We define the holonomy map $h_{VU} : D_{VU}                             \longrightarrow T_{V}$ in the following way:
   \begin{equation*}
       h_{VU}(P) = P' \ \text{if and only if the plaques} \ P \cap P' \neq \emptyset. 
   \end{equation*}
   Then, we define the \emph{Holonomy pseudogroup of $\mathcal{F}$} as the pseudogroup generated by all the maps  $h_{VU}$. 
\end{Ex}

\begin{Ex}
    If $G$ is a group finitely generated by $G_{1}$ then the pseudogroup  $\mathcal{G}(G_{1})$ generated by $G_{1}$ always contains the group $G$.
\end{Ex}

\begin{Ex}
Let $X$ be a metric space and consider $Iso(X)$ the group of isometries in $X$, that is,
\begin{equation*}
    d(g(x), g(y)) = d(x, y)
\end{equation*}
for every $g \in Iso(X)$ and $x, y \in X$.  Then, the family $\mathcal{G}({Iso}(X))$ generated by $Iso(X)$ is a pseudogroup including all of the local isometries.
\end{Ex}

For the rest of the paper we will always deal with finitely generated pseudogroups.

\subsection{Dynamical Balls}

Let $(\mathcal{G},  \mathcal{G}_{1})$ be a finitely generated pseudogroup. Define 
\begin{equation}\label{*}
\mathcal{G}_{n}^{x}=\{ g \in \mathcal{G}_{n}; \ x \in D_{g}  \}
\end{equation}
where, $\mathcal{G}_{n}  = \{  g_{1} \circ \ldots \circ g_{n}; g_{i} \in  \mathcal{G}_{1}\}$.

Now, we define the dynamical balls. For $n\in\mathbb{N}$ and $\eps> 0$, we define the \emph{dynamical} $n-$balls centered at $x$ with radius $\eps$ with  respect to the pseudogroup $(\mathcal{G}, \mathcal{G}_{1})$ given by 
\begin{equation*}
    B_{n}^{\mathcal{G}}(x, {\eps}) = \{ y \in X; \ d(g(x), g(y )) < \delta, \   \forall g \in  \mathcal{G}_{n}^{x} \cap \mathcal{G}_{n}^{y}  \},
\end{equation*}
and we denote
\begin{equation*}
    B_{n}^{\mathcal{G}}[x, {\eps}] = \{ y \in X; \ d(g(x), g(y )) \leq\delta, \   \forall g \in  \mathcal{G}_{n}^{x} \cap \mathcal{G}_{n}^{y}  \}.
\end{equation*}

Then,  consider the \emph{Bowen ball}  as the intersection of all $n-$balls denoted by

\begin{center}
    $
\begin{array}{ccl}
    \Phi_ {\delta}(x) &=& \{ y \in X; \ d(g(x), g(y )) \le \delta, \   \forall g \in  \mathcal{G}_{n}^{x} \cap \mathcal{G}_{n}^{y} \ and \ n \in  \mathbb{N}    \}\\
    \\
        & = &  \bigcap\limits_{n \in \mathbb{N}}   B_{n}^{\mathcal{G}}[x, {\eps}].
\end{array}
$

\end{center}

\begin{Lema}
     The Bowen and dynamical balls are Borel sets.
\end{Lema}
\begin{proof}
    Note  that
$$B_{n}^{\mathcal{G}}(x, {\eps})=\bigcap_{g\in \mathcal{G}_{n}^{x}}\left(g^{-1}(B(g(x),\eps))\cup D_g^c\right)=:\bigcap_{g        \in \mathcal{G}_{n}^{x}} (A_g^x\cup D_g^c).$$

Let $y\in B_{n}^{\mathcal{G}}(x,  \eps)$ and take $g\in \mathcal{G}_n^x$. Then, either $y\in D_g^c$ or $g\in G_n^y$. In the latter case, we have
$$d(g(x),g(y))<\eps.$$
Hence, $y\in g^{-1}(B(g(x),\eps))$.

Now, let $y\in \bigcap_{g\in \mathcal{G}_{n}^{x}}\left(g^{-1}(B(g(x),\eps))\cup D_g^c\right)$ and take $g\in \mathcal{G}_{n}^{x}\cap \mathcal{G}_{n}^{y}$. Then $y\in D_g$ and $d(g(x),g(y))<\eps$, hence $y  \in B_{n}^{\mathcal{G}}(x, \eps)$.

Note that, $A_g^x$ is an open set, since $D_g$ is an open set and $D_g^c$ is a closed set, which are Borel sets. Hence $B_{n}^{\mathcal{G}}(x,                \eps)$ is a Borel set, since it is a finite intersection of Borel Sets. 

Similarly, we have $B_{n}^{\mathcal{G}}[x ,\eps] = \bigcap_{g\in \mathcal{G}_{n}^{x}}\left(g^{-1}(B[g(x),\eps])\cup D_g^c\right)$. And since $B[g(x),\eps]$ is a closed subset in $X$ follows that $B[g(x),\eps] \cap R_{g}$  is a closed subset in $R_{g}$ and then by the continuity of $g$ we have $g^{-1}(B[g(x),\eps]) = g^{-1}(B[g(x),\eps] \cap R_{g})$       is a closed subset in $D_{g}$. Therefore there exists a closed subset $F_{g}$ in $X$ such that $  g^{-1}(B[g(x),\eps])  = F_{g} \cap D_{g}$, hence it is a Borel set and consequently $ B_{n}^{\mathcal{G}}[x ,\eps]  $ is a Borel set.

Finally, since $\Phi_{\eps}(x)$ is a nested and countable intersection of $ B_{n}^{\mathcal{G}}[x ,\eps]$, it is also a Borel set.
\end{proof}

\subsection{Good Generating Sets}

It turns out that due to the non compactness of the domains of the elements of the pseudogroup, it is usually required some regularity on the generating sets, see Chapter $2$ in \cite{Good}.

\begin{Def}\cite{Good} \label{good}
    Let  $\mathcal{G}$ be a pseudogroup having a finite symmetric generating set $\mathcal{G}_{1}$. The generating set $\mathcal{G}_{1}$ is said to be \emph{good} if for any $g \in \mathcal{G}_{1}$ there exists a compact subset $K_{g}  \subset D_{g}$ in such a way that 
    $$\mathcal{G}_{2} = \{  g|_{int(  K_{g})}; \ g \in \mathcal{G}_{1}, \    K_{g}                   \subset D_{g} \}$$
    still generates $\mathcal{G}$. We refer to this generating set as the \emph{compacted generating set}. Finally, The pseudogroup       $\mathcal{G}$ is called good if it admits a good generating set.
\end{Def}

\begin{Ex}
    \cite{Good} Every pseudogroup generated by a finite generating set of a group of maps of a compact space, and every                               holonomy pseudogroup associated to a nice covering of a compact foliated manifold are good.
\end{Ex}

As in Definition \ref{good}, to keep tracking the generating sets used, if $(\mathcal{G}, \mathcal{G}_{1})$ is a good pseudogroup, we will write 
\begin{equation*}
    \mathcal{G}_{n}^{1,x}:= \mathcal{G}_{n}^{x}  = \{ g \in \mathcal{G}_{n}^{1}; \ x \in D_{g}  \}
\end{equation*}
where  $\mathcal{G}_{n}^{1}= \mathcal{G}_{n}  = \{  g_{1} \circ \ldots \circ g_{n};\  g_{i} \in  \mathcal{G}_{1}\}$
and 
\begin{equation*}
    \mathcal{G}_{n}^{2,x}:=\{ g \in \mathcal{G}_{n}^{2}; \ x \in D_{g}  \}
\end{equation*}
where,  $\mathcal{G}_{n}^{2}  = \{  g_{1}' \circ \ldots \circ g_{n}';               \  g_{i}' \in  \mathcal{G}_{2}\} = \{g_{1}|_{int(K_{g_{1}})} \circ \ldots \circ g_{n}|_{int(K_{g_{n}})}; \ g_{i} \in \mathcal{G}_{1}, \ 1 \leq                       i \leq n\}$.


With this difference of the generators in mind, we will denote $\Phi_{\delta}^{1}(x) :=  \Phi_{\delta}(x)  $ and 
\begin{equation*}
    \Phi_ {\delta}^{2}(x) :=   \{ y \in X; \ d(g(x), g(y )) \le \delta, \   \forall g \in  \mathcal{G}_{n}^{2,x} \cap \mathcal{G}_{n}^{2,y} \ and \ n \in  \mathbb{N}    \}.
\end{equation*}
the Bowen ball with respect to the generating set $\mathcal{G}_{2}$.

\subsection{Ergodic Theory of Pseudogroups}

Now, we state the basic definitions of the ergodic theory for pseudogroups, based on \cite{Good}.

Let $\mathcal{G}  \subset Homeo(X)$ be a pseudogroup with finite and symmetric generating set  $\mathcal{G}_{1}$.

\begin{Def}
    A Borelian subset  $A \subset X$ is said to be  invariant under $\mathcal{G}$ if
    \begin{equation*}
              g(A \cap   D_{g} ) \subset A,  
    \end{equation*}
    for every $g   \in  \mathcal{G}.$
\end{Def}

We give a simple criterion to  a set to be invariant.

\begin{Prop}\label{.}
         Let $\mathcal{G}$ be a finitely generated pseudogroup with a generating set      $\mathcal{G}_1=  \{ g_{1}, \ldots, g_n \}$. If a subset $A \subset X$ is  $\mathcal{G}_{1}-$invariant then $A$ is also $\mathcal{G}-$invariant, that is,
    \begin{equation*}
        g_{i}( A \cap D_{g_{i}}) \subset A,  \ \forall g_{i}   \in \mathcal{G}_{1} \Rightarrow  g( A \cap D_{g}) \subset A,  \ \forall g   \in \mathcal{G}.
    \end{equation*}

    \end{Prop}

\begin{proof}
     Let  $g \in \mathcal{G}$. Take  $ x  \in A \cap D_{g}$. Then,  there exist  $g_{1}, \ldots, g_{l}$ $\in  \mathcal{G}_{1}$    such that
     \begin{equation*}
         g(x)     = g_{1}\circ \ldots \circ g_{l}(x).
     \end{equation*}

     Since   $A$ is $\mathcal{G}_{1}-$invariant, that is, $  g_{i}( A \cap D_{g_{i}}) \subset A,  \ \forall g_{i}   \in \mathcal{G}_{1}$, then 
     \begin{center}
      $
      \begin{array}{ccl}
          x \in A     & \Rightarrow & g_{l}(x) \in A\\
          \\
          &      \Rightarrow   &             g_{l-1} \circ g_{l}(x)     \in A\\
          \\
          &  \vdots & \\
          \\
          &  \Rightarrow    & g(x) = g_{1} \circ \ldots  \circ         g_{l}(x)  \in A. 
      \end{array}
      $
     \end{center}
    Therefore, $  g( A \cap D_{g}) \subset A,  \ \forall g   \in \mathcal{G}$.  
\end{proof}

\begin{mar}
    Note that since we are considering only symmetric generating sets $\mathcal{G}_{1}$, we have that a subset $A $ is $\mathcal{G}_{1}-$invariant if, and only if, 
    \begin{equation*}
             g_{i}^{-1}(A  \cap D_{g_{i}^{-1}})   \subset  A,       \              \forall g_{i} \in \mathcal{G}_{1}.
    \end{equation*}
\end{mar}

\begin{Def}
A probability measure $\mu$ in $X$ is said to be invariant under $\mathcal{G}$ if 
\begin{equation*}
   \mu(g(A)) = \mu(A)
\end{equation*}
for every $ g \in \mathcal{G}$ and every Borelian  subset $A$ of $D_{g}$.

\end{Def}

\begin{Def}
   An invariant probability   measure $\mu$ in X is said to be ergodic if every invariant subset  $A \subset X$ is such that $\mu(A)=0 $ or 1.
\end{Def}

We also recall the notion of entropy using dynamical balls as in \cite{AB}.

\begin{Def}
    Let $x \in X$ and $\mu$ be a probability measure  in  X. The numbers

    \begin{equation*}
         \underline{h}_{\mu}((\mathcal{G}, \mathcal{G}_{1}),x)  = \lim_{  \varepsilon \to  0} \liminf_{n \to  \infty }  \frac{-1}{n}  \log \mu(B_{n}^{\mathcal{G}}(x, \varepsilon))
    \end{equation*}

    and

\begin{equation*}
        \overline{h}_{\mu}((\mathcal{G}, \mathcal{G}_{1}),x)  = \lim_{  \varepsilon \to  0} \limsup_{n \to  \infty }  \frac{-1}{n}  \log \mu(B_{n}^{\mathcal{G}}(x, \varepsilon))
    \end{equation*}
    
are called the local lower $\mu-$measure entropy  at the point $x$ and local upper $\mu-$measure entropy  at the point $x$                                        with respect to $(\mathcal{G}, \mathcal{G}_{1})$, respectively.

\end{Def}

Next, we define the topological entropy of a finitely generated pseudogroup, introduced in \cite{Entropy}.

\begin{Def}
Let      $(\mathcal{G}, \mathcal{G}_{1})$ be a finitely generated pseudogroup.                                                                       Given $\eps > 0$ and $n \in \mathbb{N}$,  two points $x,y \in X$ are  said to be $((\mathcal{G}, \mathcal{G}_{1}), n, \eps)-$separated if there exists $   g \in \mathcal{G}_{n}^{x} \cap    \mathcal{G}_{n}^{y}$ such that                            $d(g(x), g(y))                   \geq \eps$.

A subset $E \subset X$ is said to be a $((\mathcal{G}, \mathcal{G}_{1}), n, \eps)-$separated set if any two distinct points $x, y \in E$ are $((\mathcal{G}, \mathcal{G}_{1}), n, \eps)-$separated. 

We denote                           $s((\mathcal{G}, \mathcal{G}_{1}), n, \eps)$ the maximal cardinality of a $((\mathcal{G}, \mathcal{G}_{1}), n, \eps)-$separated subset of X.
\end{Def}

These numbers are finite by the compactness of $X$. The topological entropy then measures the exponential growth rate of such numbers at infinitesimal scales.

\begin{Def}
Let $ (\mathcal{G}, \mathcal{G}_{1}) $   be  a finitely generated pseudogroup. The topological entropy of  $(\mathcal{G}, \mathcal{G}_{1})$ is defined by

\begin{equation*}
    h_{top}( \mathcal{G}, \mathcal{G}_{1}) = \lim_{\eps \rightarrow       0^{+}    }  \limsup_{n \rightarrow +\infty}           \frac{1}{n} \log s((\mathcal{G},            \mathcal{G}_{1}), n, \eps)
\end{equation*}
\end{Def}

\section{Expansive Measures}\label{s.exp}

In this section, we define and explore the notion of expansive measures.

\begin{Def}
     Let  $\mathcal{G}$ be a finitely generated pseudogroup with generating set $\mathcal{G}_{1}$. A    Borel               probability measure  $\mu$ in X  is said to be  $(\mathcal{G}, \mathcal{G}_{1})$-expansive  if there exists  $\delta > 0$ such that 
\begin{equation*}
        \mu( \Phi_ {\delta}(x)) =0
    \end{equation*}
for every  $  x \in X$. When the pseudogroup and the generating set are well understood, we simply say that the measure is expansive.

\end{Def}

We remark that expansive measures are always non-atomic.

 In the same spirit, as we did before, we say that a Borel    probability measure $\mu$ is expansive for a good pseudogroup                   $\mathcal{G}$ with respect to $\mathcal{G}_{2}$,  if there exists $\delta  > 0$ such that 
\begin{equation*}
    \mu( \Phi_ {\delta}^{2}(x)) = 0
\end{equation*}
for every                 $x \in X$.

Now, we will discuss the relations between such definitions with respect to the generating set.

\begin{Prop}\label{Bal}
    Let          $\mathcal{G}$ be a good pseudogroup, $\mathcal{G}_{1}$ a good generating set and $\mathcal{G}_{2}$  the compacted generating set. Then, for every $\eta  > 0$                    and $x \in      X,$ we have 
    \begin{equation*}
        \Phi_{\eta}^{1}(x) \subset \Phi_{\eta}^{2}(x).
    \end{equation*}
  
\end{Prop}

\begin{proof}
    Let $y \in          \Phi_{\eta}^{1}(x)$ and $g' \in \mathcal{G}_{n}^{2,x} \cap  \mathcal{G}_{n}^{2,y}$, for some $n \in \mathbb{N}$.

Then, by definition there are $g_{1}, \ldots, g_{n} \in \mathcal{G}_{1}$ such that 
$$g' = g_{1}|_{int(K_{g_{{1}}})} \circ \ldots \circ g_{n}|_{int(K_{g_{{n}}})}\textrm{ and }x, y \in D_{g'}.$$
Hence,
\begin{equation*}
                           d(g(x)), g(y)) \leq \eta 
\end{equation*}
   for every $g \in \mathcal{G}_{n}^{1,x} \cap \mathcal{G}_{n}^{1,y}$ and $n           \in \mathbb{N}$. In particular, this is true for the extension $ g = g_{1} \circ    \ldots \circ g_{n}$ of $g'$, because $x, y \in D_{g_{1}|_{int(K_{g_{{1}}})} \circ \ldots \circ g_{n}|_{int(K_{g_{{n}}})}} \subset D_{g_{1}  \circ \ldots \circ g_{n}}$, therefore, 
\begin{equation*}
    d(g'(x), g'(y)) = d(g(x), g(y)) \leq \eta,
\end{equation*}
showing that $y \in                               \Phi_{\eta}^{2}(x)$.
\end{proof}

\begin{Cor}
Let          $\mathcal{G}    $ a good pseudogroup, $\mathcal{G}_{1}$ a good generating set and $\mathcal{G}_{2}$  the compacted generating set.   If $\mu$ is       $(\mathcal{G}, \mathcal{G}_{2})$-expansive   then          $\mu$ is also  $(\mathcal{G}, \mathcal{G}_{1})$-expansive.
\end{Cor}

\begin{proof}
    By Proposition \ref{Bal}     we have $\Phi_{\eta}^{1}(x) \subset \Phi_{\eta}^{2}(x)$ then $\mu(\Phi_{\eta}^{1}(x))               \leq  \mu(\Phi_{\eta}^{2}(x)) = 0,$ for every $x \in X$.
\end{proof}

We can also control that inclusion when the points are different.

\begin{Lema}\label{9}
Let          $\mathcal{G}$ a good pseudogroup, $\mathcal{G}_{1}$ a good generating set and $\mathcal{G}_{2}$  the compacted generating set.  Then,                  there exists $\rho > 0$ such that 
if     $y_{0} \in \Phi_{\frac{\rho}{2}}^{1}(x_{0}) $ then $\Phi_{\frac{\rho}{2}}^{1}(x_{0})  \subset  \Phi_{\rho}^{2}(y_{0}).
$
\end{Lema}

\begin{proof}
    Firstly, note that for each $g \in \mathcal{G}_{1}$ there exists $\rho_{g} > 0$ such that if $   z  \notin D_{g}$            and $y \in int(K_{g})$ then $d(z, y) > \rho_{g}$. Hence, since $\mathcal{G}_{1}$ is finite, take $\rho = \min \{\rho_{g}; \ g \in \mathcal{G}_{1}\} > 0$.

    Now, let $x                      \in            \Phi_{\frac{\rho}{2}}^{1}(x_{0})$ and $g' \in \mathcal{G}_{n}^{2,y_{0}} \cap \mathcal{G}_{n}^{2,x}$. Thus, there exist $g_{1}, \ldots, g_{n}        \in \mathcal{G}_{1}$ such that 
\begin{equation*}
    g' = g_{1}|_{int(K_{g_{1}})} \circ \ldots \circ g_{n}|_{int (K_{g_{n}})}.
\end{equation*}
Then, take $g =             g_{1} \circ \ldots \circ g_{n}$ the extension of $g'$. Clearly $x, y_{0} \in D_{g}$.
\\

\textbf{Claim:} $x_{0} \in  D_{g}$.

\emph{Proof of the Claim.}

Suppose by contradiction that $x_{0}                   \notin D_{g}$. So, either $x_{0}                 \notin D_{g_{n}}$, or $g_{n}(x_{0})                 \notin D_{g_{n}}, \ldots$, or $g_{2}                                       \circ \ldots                                                           \circ g_{n}(x_{0})                 \notin D_{g_{1}}$.

Since $y_{0} \in     D_{g'}$, we have $$y_{0}         \in int(K_{g_{n}}),  g_{n}(y_{0}) \in int(K_{g_{n-1}}), \dots, g_{2} \circ \ldots \circ                        g_{n}(y_{0}) \in int( K_{g_{1}}).$$


Suppose first that $x_{0} \notin D_{g_{n}}$. Since                      $y_{0} \in int(K_{g_{n}})$, by the choice of $\rho$ we have 
\begin{equation}\label{e}
    d(x_{0},  y_{0})  > \rho.
\end{equation}
But, $y_{0} \in \Phi_{\frac{\rho}{2}}^{1}(x_{0})$   and $id_{X}      \in \mathcal{G}_{n}^{1,  x_{0}} \cap \mathcal{G}_{n}^{1,  y_{0}}$,                   then 
\begin{equation*}
    d(x_{0}, y_{0}) \leq \frac{\rho}{2},
\end{equation*}
which is a contradiction with (\ref{e}).

Hence, we can suppose that $x_{0} \in D_{g_{n}}$. Now, if $g_{n}(x_{0}) \notin D_{g_{n-1}}$, then, since $g_{n}(y_{0})  \in \  int(        K_{g_{n-1}})$, we have 
\begin{equation}\label{z}
    d(g_{n}(x_{0}), g_{n}(y_{0}))  > \rho.
\end{equation}
But, since $y_{0} \in \Phi_{\frac{\rho}{2}}^{1}(x_{0}),$ and $g_{n} \in \mathcal{G}_{1}^{1, x_{0}} \cap \mathcal{G}_{1}^{1, y_{0}}$, we have 
\begin{equation*}
    d(g_{n}(x_{0}),        g_{n}(y_{0})) \leq \frac{\rho}{2}
\end{equation*}
             which is a contradiction with (\ref{z}).

             Therefore, repeating the argument in the same way, by induction, we complete the proof of the claim.

     To finish the proof of the Lemma, since $g'(x)  = g(x)$, $g'( y_{0}   ) = g(y_{0})$,      $x, y_{0}, x_{0} \in D_{g}$ and $g \in \mathcal{G}_{n}^{1}$ then

     \begin{center}
         $
         \begin{array}{ccl}
              d(g'(x), g'(y_{0})) =   d(g(x), g(y_{0}))&\leq&  d(g(x), g(x_{0})) + d(g(x_{0}), g(y_{0}))\\
              \\
              & \leq & \frac{\rho}{2}  + \frac{\rho}{2}= \rho.
         \end{array}
         $
     \end{center}
        Therefore, $x \in \Phi_{\rho}^{2}(y_{0})$. The proof of the Lemma is complete. 
     \end{proof}

For homeomorphisms, it was proved in \cite{AM} that to prove that a measure is expansive is enough to consider only the Bowen balls for almost every point. However, it is much delicate when consider pseudogroups. This motivates the following definition.

\begin{Def}
    Let $(\mathcal{G}, \mathcal{G}_{1})$ be a finitely generated pseudogroup. We say that a Borel probability measure $\mu$ in $X$ is  $(\mathcal{G},\mathcal{G}_{1})-$weakly expansive, if there exists $\delta  > 0$ such that 
    \begin{equation*}
        \mu(                        \Phi_{\delta}(x)) = 0,
    \end{equation*}
    for $ \mu-a.e. \  x                                     \in X$.
\end{Def}

It is trivial to prove that every expansive measure is weakly expansive. For the converse, we need to recover some compactness of the domains of the generating set. We then recast Theorem A with all of its details as follows.

     \begin{Teo}\label{qtp}
Let          $\mathcal{G}$ a good pseudogroup, $\mathcal{G}_{1}$ a good generating set and $\mathcal{G}_{2}$ the compacted generating set. Then, there exists $\rho> 0$ such that  if a measure          $\mu$ is $(\mathcal{G}, \mathcal{G}_{2})$-weakly expansive with constant                    $\rho > 0$     then $\mu$ is $(\mathcal{G}, \mathcal{G}_{1})$-expansive with constant $\frac{\rho}{2}$.
     \end{Teo}

    \begin{proof}
        
    Take $\rho > 0 $ as in Lemma                     \ref{9} and suppose   by contradiction that there exists $x_{0} \in  X$ such that $\mu(\Phi_{\frac{\rho}{2}}^{1}(x_{0})) >    0$. Since $X_{\rho} = \{x      \in X; \ \mu(\Phi_{\rho}^{2}(x)) = 0\}$ has full measure, we have $X_{\rho} \cap \Phi_{\frac{\rho}{2}}^{1}(x_{0}) \neq \emptyset$. Take $y_{0} \in X_{\rho} \cap \Phi_{\frac{\rho}{2}}^{1}(x_{0})$.

By Lemma \ref{9}, we obtain that  $\Phi_{\frac{\rho}{2}}^{1}(x_{0}) \subset \Phi_{\rho}^{2}(y_{0})$. Then, $\mu(\Phi_{\frac{\rho}{2}}^{1}(x_{0}))                                \leq \mu(\Phi_{\rho}^{2}(y_{0})) = 0$. But, this is a contradiction because $\mu(\Phi_{\frac{\rho}{2}}^{1}(x_{0}))$ was positive.
\end{proof}


The following questions arise naturally from this discussion.

\textbf{Question A:}   Does  expansiveness for pseudogroup always coincide with a.e.-expansiveness for the same generating set?

\textbf{Question B:}  Does expansiveness for pseudogroup independ of the      generating set?

Now, we deal with measure expansive and countable expansive pseudogroups.

\begin{Def}
    We say that a finitely generated pseudogroup $(\mathcal{G}, \mathcal{G}_{1} )$ is countably-expansive if there is $\delta > 0$ such that $\Phi_{\delta}(x)$                                   is                          countable, for every $x           \in X$.
\end{Def}

\begin{Def}

A finitely generated pseudogroup $(\mathcal{G}, \mathcal{G}_{1})$ is said to be measure-expansive if every non-atomic Borel probability measure $\mu$ is $(\mathcal{G}, \mathcal{G}_{1})$-expansive. 
Analogously, $(\mathcal{G}, \mathcal{G}_{1})$ is  weakly-measure expansive  if every non-atomic Borel                                      probability measure $\mu$ is $(\mathcal{G}, \mathcal{ G}_{1})$-weakly expansive.
\end{Def}
The next proposition recast the result due to Artigue and Dante \cite{artig} for pseudogroups.
\begin{Prop}\label{q}
    Let $(\mathcal{G}, \mathcal{G}_{1})$ be a finitely generated pseudogroup. Then, $(\mathcal{G}, \mathcal{G}_{1})$ is measure-expansive if, and only if,          $(\mathcal{G}, \mathcal{G}_{1})$ is countably-expansive.
\end{Prop}

\begin{proof}
    The proof for pseudogroups keeps similar         to the one in \cite{artig} for homeomorphisms.
\end{proof}

However, if the pseudogroup is good we can relate it with weak-expansivity.

\begin{Prop}\label{cou}
    Let $( \mathcal{G},                       \mathcal{G}_{1})$ be a good pseudogroup and $\mathcal{G}_{2}$  the compacted generating set.   If $(\mathcal{G}, \mathcal{G}_{2})$ is weakly-measure expansive then $(\mathcal{G},                                    \mathcal{G}_{1})$ is countably-expansive.
\end{Prop}
\begin{proof}
    Since $(\mathcal{G}, \mathcal{G}_{2})$ is weakly-measure-expansive, that is, $(\mathcal{G}, \mathcal{G}_{2})$ is $\mu$-weakly expansive for every non-atomic Borel probability measure $\mu$, 
    by  Theorem \ref{qtp} every such measures are $(\mathcal{G}, \mathcal{G}_{1})-$expansive, hence $(\mathcal{G}, \mathcal{G}_{1})$ is measure-expansive and by Propositon \ref{q} we have $(\mathcal{G},  \mathcal{G}_{1})$ is countably-expansive.
\end{proof}

As a consequence, we obtain that countable expansivity independs of the compacted generating set.

\begin{Cor}
    Let $\mathcal{G}_{1}$ and $\mathcal{G}_{1}'$ be two good generating sets for     a     pseudogroup $\mathcal{G}$ and $\mathcal{G}_{2}$, $\mathcal{G}_{2}'$ the compacted generating sets. Then, $(\mathcal{G}, \mathcal{G}_{2})$ is countably-expansive if, and only if, $(\mathcal{G}, \mathcal{G}_{2}')$ is countably-expansive.
\end{Cor}

\begin{proof}
    By argument in Lemma 2.4.3 in \cite{P} we have that $\Phi^{2}_{\delta}(x) =         \Phi^{2'}_{\delta}(x)$ for $\delta <     \lambda$. Then,       the result follows. 
\end{proof}
In particular, we obtain the following independence result.
\begin{Cor}
 Measure-expansiveness of a pseudogroup independs on the compacted generating set.
\end{Cor}

Once again we have the following questions.

\textbf{Question C:} Is expansiveness independent of any kind of generating set?

\textbf{Question D:}   If $\Phi_{\delta}^{1}(x)$ is countable then is                 $\Phi_{\delta}^{2}(x)$  also countable?

\section{Conjugacy Properties}
\label{s.iso}

Let      $(X, d_{X})$ and                               $(Y, d_{Y})$
be compact metric spaces and denote by $ Homeo(X, Y)$ the set of all homeomorphims  between open subsets of $X$ and $Y$ respectively.

\begin{Def}\label{def}\cite{Good}
   Let $\mathcal{G} \subset Homeo(X)$  and $\mathcal{H} \subset  Homeo(Y)$ be finitely generated pseudogroups. A subset $\Phi \subset  Homeo(X, Y)$ is said to be an isomorphism between $\mathcal{G}$      and     $\mathcal{H}$, if 
\begin{equation*}
       \bigcup_{\phi \in \Phi} D_{\phi} = X, \ \bigcup_{\phi \in \Phi} R_{\phi} = Y
   \end{equation*}

   \begin{equation*}
       \phi \circ f \circ \psi^{-1} \in \mathcal{H} \Longleftrightarrow f \in \mathcal{G},
   \end{equation*}
   for $ \phi,  \psi\in \Phi$.
\end{Def}

In this case, the pseudogroups $\mathcal{G},  \mathcal{H}$ are called isomorphic and  the set $\Phi$ is                    sometimes denoted as $\Phi: \mathcal{G}   \longrightarrow \mathcal{H}$. Further,  if $\Phi:  \mathcal{G}  \longrightarrow \mathcal{H}   $ is an isomorphism and every map in $\Phi$ is uniformly continuous then $\Phi$ is called an uniform isomorphism, and $\mathcal{G}$  and     $\mathcal{H}$ are                 said to be uniformly isomorphic.

\textbf{Remark:} Observe that since $X$ and $Y$ are compact,    if $\mathcal{G}$ and $\mathcal{H}$ are isomorphic, then $\Phi$ may be taken finite.

The next lemma produces a generating set for a pseudgroup under the action of an isomorphism.

\begin{Lema}\label{A}
    Let $\mathcal{G} \subset Homeo(X)$            and $\mathcal{H} \subset  Homeo(Y)$ be isomorphic finitely generated pseudogroups, with an isomorphism ${\Phi} = \{ \varphi_{1}, \ldots , \varphi_{l}\}$. If                      $\mathcal{G}_{1}$ is a generating set for $\mathcal{G}$,                                   then $\mathcal{H}_{1}:= \{\phi_j\circ f \circ \phi_i^{-1}; \ f\in           \mathcal{G}_{1}, i,j=1,\dots,l\ \}$ is  a generating set for $\mathcal{H}$.
\end{Lema}

\begin{proof}
    Let $g \in \mathcal{H}$ and take $x' \in D_{g}$. There exists $p$ such that $x' \in R_{\phi_{p}}$. So, there exists $x \in D_{\phi_{p}}$             such that $\phi_{p}(x) = x'$. In the same manner, there exists $q$, such that $g(x')\in R_{\phi_{q}}$.  Let $\theta:=\phi_q$ and $\psi:=\phi_p$.

Let $f=\theta^{-1}\circ g  \circ \psi\in \mathcal{G}$. Notice that $f(x)=\theta^{-1}(g(x'))$, so $f(x)\in D_{\theta}$.

So there is      an open neighborhood $U$ of $x$ and $f_1,\dots , f_n$ in $\mathcal{G}_1$ such that
\begin{equation*}
f|_U=f_1\circ \dots \circ   f_n|_U.
\end{equation*}
Hence, $x\in D_{\phi_{p}}\cap U\cap D_{f_n}.$ Let $x_n=f_n(x)$.

There exists $i_n$, such that $x_n\in D_{\phi_{i_{n}}}\cap D_{f_{n-1}}$. Let $x_n'=\phi_{i_n}(x_n)\in R_{\phi_{i_{n}}}$. We define $g_n=\phi_{i_n}\circ f_n\circ \psi^{-1}\in \mathcal{H}_1$. Notice that $x'  \in D_{g_{n}}$, $g_n(x')=x_n'$ and $V=  \psi(U)$ is open.

Now, we proceed by induction. Let $x_{n-1}=f_{n-1}(x_n)$. So there exists $i_{n-1}$ such that $x_{n-1}\in D_{\phi_{i_{n-1}}}\cap D_{f_{n-2}}$. Let $x_{n-1}'=\phi_{i_{n-1}}(x_{n-1})\in R_{\phi_{i_{n-1}}}$.

Now, we define $g_{n-1}=\phi_{i_{n-1}}\circ f_{n-1}\circ \phi_{i_n}^{-1}\in \mathcal{H}_1$. Notice that 
$$g_{n-1}\circ g_n=\phi_{i_{n-1}}\circ f_{n-1}\circ \phi_{i_n}^{-1}\circ \phi_{i_n}\circ f_n\circ \psi^{-1}=\phi_{i_{n-1}}\circ f_{n-1}\circ f_n\circ \psi^{-1}$$
is defined on $V$ and $g_{n-1}\circ g_n(x')=x_{n-1}'$.

So, by induction, we construct $g_1,\dots, g_n\in \mathcal{H}_1$ such that $g_1\circ \dots \circ g_n$ is defined on $V$. 

However, notice that $x_1=f_1\circ\dots\circ f_n(x)=f(x)$. Hence, we can choose $\phi_{i_1}=\theta$.

Thus,

  \begin{center}
        $
        \begin{array}{ccl}
            g|_{V} &=& \theta \circ f\circ \psi^{-1}|_{V}
    =\phi_{q} \circ f_{1} \circ \ldots \circ          f_{n}|_{U}=\\
    \\
     &=& (\phi_{i_{1}} \circ f_{1}) \circ \ldots \circ          f_{n}|_{U}=(g_{1}      \circ \phi_{i_{2}}) \circ f_{2} \circ \ldots f_{n}|_{U}=\\
    \\
    &    = &  g_{1} \circ g_{2} \circ  \phi_{i_{3}}           \circ         \ldots \circ f_{n}|_{U}=\\  
    \\
    && \vdots
    \\
    &=&   g_{1} \circ g_{2} \circ \ldots   \circ  \phi_{i_{n}} \circ f_{n}|_{U}=g_{1} \circ g_{2}             \circ \ldots \circ g_{n}                       \circ \phi_{p}|_{U}=\\
    \\
      & = & g_{1} \circ g_{2}             \circ \ldots \circ g_{n}|_{\phi_{p}(U)}.
        \end{array}
        $
    \end{center}

This shows that $\mathcal{H}_1$ is a generator.
\end{proof}




\begin{Def}
If $\Phi$ reduces to a single  map $\varphi:X \longrightarrow Y$     satisfying definition \ref{def}, then we say the pseudogroups are  strongly isomorphic.
\end{Def}

Now, we see that any global homeomorphism induces a new pseudogroup from the initial one.

\begin{Prop}
    Let                $\mathcal{G}$ be  a      pseudogroup in $X$ and $      \varphi: X \longrightarrow  Y$ an homeomorphism. Then, the subset $\mathcal{H}   = \varphi \circ \mathcal{G}   \circ \varphi^{-1} := \{ \varphi \circ g \circ \varphi^{-1}; \ g \in \mathcal{G}\}$ is a pseudogroup in         $Y$.
\end{Prop}

\begin{proof}
    First, note that for each $g \in \mathcal{G}$, $\varphi \circ g \circ \varphi^{-1}$ is an homeomorphism,   whose domain and image are the open sets
$$D_{ \varphi  \circ g \circ\varphi^{-1}}=\varphi(D_{g})\textrm{ and } R_{\varphi          \circ g \circ \varphi^{-1}}= \varphi( R_{g}).
$$

Therefore, $\mathcal{H}  \subset Homeo(Y)$. 
Now, we prove that $\mathcal{H}$ is actually a pseudogroup.

    \begin{enumerate}[label=(\arabic*)]

    \item If $h_{1} =    \varphi \circ g_{1} \circ \varphi^{-1}$ and $h_{2} =                  \varphi \circ g_{2} \circ \varphi^{-1} \in \mathcal{H}$, for $g_{1}, g_{2} \in \mathcal{G}$, then 
    $h_{1}  \circ h_{2}  =       \varphi \circ g_{1} \circ g_{2} \circ \varphi^{-1}
    $ belongs to $\mathcal{H}$, because $g_{1} \circ g_{2} \in \mathcal{G}$.
    \\
\item If $h = \varphi \circ g \circ \varphi^{-1} \in \mathcal{H}$, for                some $g \in \mathcal{G}$,       then  $h^{-1}  = \varphi \circ g^{-1}                 \circ \varphi^{-1} \in    \mathcal{H}$, since   $g^{-1} \in \mathcal{G}$.
\\
\item  If $h  \in \mathcal{H}$ and $U \subset D_{h}$ is an open subset, then  $h|_{U} = (\varphi \circ   g \circ \varphi^{-1})|_{U} = \varphi \circ g|_{\varphi^{-1}(U)} \circ \varphi^{-1}$. Therefore, $h|_{U} \in \mathcal{G}$, since $g|_{\varphi^{-1}(U)} \in \mathcal{G}$.
\\
        \item   Suppose $h \in Homeo(Y)$,   $\mathcal{U}$ is an open covering of $D_{h}$, and $h|_{U} \in \mathcal{H},$ for every $                 U \in \mathcal{U}$.

        Then, for each $U \in \mathcal{U}$ there is $g_{U}  \in \mathcal{G} $ such that $h|_{U}  = \varphi   \circ g_{U}  \circ \varphi^{-1}$, and, $g_{U}|_{\varphi^{-1}(U)} \in \mathcal{G}$ for each $ U \in \mathcal{U}$.\\
        Hence, $h = \varphi  \circ  g  \circ  \varphi^{-1}$, where $g: \bigcup_{U \in \mathcal{U}} \varphi^{-1}(U) \longrightarrow \bigcup_{U      \in                            \mathcal{U}} g \circ \varphi^{-1}(U)$ is given by $g(x) = g_{U}|_{\varphi^{-1}(U)}(x),$ if $x \in \varphi^{-1}(U)$. 

        \textbf{Claim:} $g$ is well defined.

        Indeed, let $x \in \varphi^{-1}(U) \cap \varphi^{-1}(V)$, for               $U, V \in \mathcal{U}$, then $\varphi(x) \in U \cap V$ and
$h|_{U}(\varphi(x))            =  h|_{V}(\varphi(x))$, thus $$            \varphi    \circ  g_{U}|_{\varphi^{-1}(U)} \circ \varphi^{-1}(\varphi(x))= \varphi        \circ g_{V}|_{\varphi^{-1}(V)}          \circ \varphi^{-1}(\varphi(x)).$$
        implying that $g_{U}|_{\varphi^{-1}(U)}(x)  = g_{V}|_{\varphi^{-1}(V)}(x)$. Therefore, $g$ is well defined.

        Since $g|_{\varphi^{-1}(U)} \in \mathcal{G}$, for each $U$,    then $g \in \mathcal{G}$,  and  $ h \in \mathcal{H}$.
\\
        \item $id_{Y} \in  \mathcal{H}$, because $id_{Y} = \varphi \circ id_{X}   \circ            \varphi^{-1}$.
    \end{enumerate}
    Then,  $\mathcal{H}$ is a pseudogroup in $Y$.
\end{proof}

In particular, we can control the generating sets under such isomorphism.

\begin{Cor}\label{Z}
     Let $\mathcal{G} \subset \ Homeo(X)$            and $\mathcal{H} \subset \ Homeo(Y)$  strongly isomorphic finitely generated pseudogroups, with an isomorphism    $\varphi  : X \longrightarrow Y$. Hence, if                      $\mathcal{G}_{1}$ is a generating set for $\mathcal{G}$,                                   then $\mathcal{H}_{1}= \{\varphi  \circ         g \circ          \varphi^{-1}; \ g\in           \mathcal{G}_{1}  \}$ is  a generating set for $\mathcal{H}$.
\end{Cor}
\begin{proof}
    Immediate by Lemma \ref{A}.
\end{proof}

In what follows we will consider $(\mathcal{G}, \mathcal{G}_{1})$ and                                 $(\mathcal{H}, \mathcal{H}_{1})$ with generating sets like in Corollary \ref{Z}.

\begin{Teo}\label{iso}
 Let $(\mathcal{G}, \mathcal{G}_{1})$ and  $(\mathcal{H}, \mathcal{H}_{1})$ strongly isomorphic pseudogroups, with a isomorphism $   \varphi : X \longrightarrow Y$. If $\mu \in M(X)$   is $(\mathcal{G},\mathcal{G}_{1})$-expansive with a constant $\eta$, then there exists $\delta > 0$ such that $  \varphi_{*}\mu \in M(Y)$ is $(\mathcal{H}, \mathcal{H}_{1})$-expansive with constant $\delta$.    
\end{Teo}

\begin{proof}
       Let   $\eta$ the expansiveness constant of $\mu$. Since $X$ and $Y$ are compact and $   \varphi$ is an homeomorphism, follows that $\varphi^{-1}$ is uniformly continuous, that is, for $\eta > 0$ there exists $\delta' > 0$ such that 
       \begin{equation*}
           d_{Y}(x',y') < \delta' \Rightarrow d_{X}(\varphi^{-1}(x'),  \varphi^{-1}(y'))  < \eta.
       \end{equation*}
       Take $                                               \delta                      < \delta'$ and let $x \in X$ any.

       \textbf{Claim:}  $\varphi^{-1}(\Phi_{\delta}(\varphi(x)))  \subset \Phi_{\eta}(x)$.

       Indeed, let               $z \in               \varphi^{-1}(\Phi_{\delta}(\varphi(x)))$ and                 $ g \in \mathcal{G}_{n}^{x}  \cap \mathcal{G}_{n}^{z}$, for          some $   n \in \mathbb{N}$.

       Take $h = \varphi \circ g \circ \varphi^{-1}$ and note that $h \in \mathcal{H}_{n}^{\varphi(x)}  \cap \mathcal{H}_{n}^{\varphi(z)}$, by Corollary \ref{Z}. Hence,  since $\varphi(z) \in \Phi_{\delta}(\varphi(x))$,           we have,
\begin{center}
    $
    \begin{array}{ccl}
         d(h(\varphi(x)),  h(\varphi(z))) \leq \delta          & \Rightarrow   & d(\varphi \circ g \circ      \varphi^{-1}(\varphi(x)),                                   \varphi \circ g \circ \varphi^{-1}( \varphi(z))) \leq \delta \\
         \\
           &  \Rightarrow                                & d(\varphi  \circ   g(x),  \varphi \circ g(z)) \leq \delta  < \delta'\\
           \\
           &   \Rightarrow           & d(\varphi^{-1} \circ \varphi \circ g(x), \varphi^{-1}  \circ \varphi \circ g(z))                        < \eta\\
           \\
           &  \Rightarrow       & d(g(x), g(z))  < \eta
\end{array}
    $
\end{center}

Hence, by the arbitrariness of $g$ follows that $z \in    \Phi_{\eta}(x)$, completing the claim.

Therefore, by the above claim, we have 
$$
\varphi_{*}\mu(\Phi_{\delta}(\varphi(x))) = \mu(\varphi^{-1}(\Phi_{\delta}(\varphi(x))))  \leq \mu(\Phi_{\eta}(x)) = 0
$$
for every $x \in X$.
So, $\varphi_{*}\mu$ is   expansive with constant $\delta$.
\end{proof}

We also can compare the entropy of measures under strong isomorphisms.

\begin{Prop}\label{ent}
    Let        $(\mathcal{G},\mathcal{G}_{1})$ and $(\mathcal{H}, \mathcal{H}_{1})$ strongly isomorphic finitely generated pseudogroups with an isomorphism $\varphi$. Then, 
  \begin{equation*}
           \underline{h}_{\mu}((\mathcal{G}, \mathcal{G}_{1}),x)= \underline{h}_{ \varphi_{*}\mu}((\mathcal{H}, \mathcal{H}_{1}),
        \varphi(x)),
    \end{equation*}
    for  every $x \in     X.$
\end{Prop}

\begin{proof}
    Similarly as proved before, for every $\eps  > 0$ there exists              $\delta = \delta(\eps) > 0$ such that $\varphi^{-1}(B_{n}^{\mathcal{H}}(\varphi(x)),            \delta))   \subset                            B_{n}^{\mathcal{G}}(x,   \eps)$. Then, 
    we have 
\begin{equation*}
     \mu(\varphi^{-1}(B_{n}^{\mathcal{H}}(\varphi(x)),            \delta))) \leq                            \mu(B_{n}^{\mathcal{G}}(x,   \eps))
\end{equation*}
       thus  
         \begin{equation*}
                  \frac{-1}{n} \log \varphi_{*}\mu(B_{n}^{\mathcal{H}}(\varphi(x)),            \delta))) \geq  \frac{-1}{n}  \log \mu(B_{n}^{\mathcal{G}}(x,   \eps))
         \end{equation*}
         so
         \begin{equation*}
          \underline{h}_{\varphi_{*}\mu}((\mathcal{H}, \mathcal{H}_{1}),                                                  \varphi(x))  \geq  \underline{h}_{\mu}((\mathcal{G}, \mathcal{G}_{1}),x).
  \end{equation*}
\vspace{0.3cm}

Analogously, arguing with $\varphi^{-1}$ instead of $\varphi$, we see that for every $\eps>0$ there exists a $\delta>0$ such that $B_{n}^{\mathcal{G}}(x,   \delta) \subset \varphi^{-1}(B_{n}^{\mathcal{H}}(\varphi(x),            \eps))                              $.

%

Hence,  we conclude that $\mu( B_{n}^{\mathcal{G}}(x,   \delta)) \leq \mu(\varphi^{-1}(B_{n}^{\mathcal{H}}(\varphi(x), \eps))) $. Hence,
$$\underline{h}_{\mu}((\mathcal{G}, \mathcal{G}_{1}),x) \geq \underline{h}_{\varphi_{*}\mu}((\mathcal{H}, \mathcal{H}_{1}),
        \varphi(x)).$$
The proposition follows.

\end{proof}

\begin{mar}
    The same result follows for the local upper measure entropy, with the same proof.
\end{mar}

We also have a similar result for the topological entropy.

\begin{Prop}\label{top}
     Let        $(\mathcal{G},\mathcal{G}_{1})$ and $(\mathcal{H}, \mathcal{H}_{1})$ strongly isomorphic finitely generated pseudogroups with an isomorphism $\varphi$. Then,
\begin{equation*}
    h_{top}(\mathcal{G}, \mathcal{G}_{1}) =  h_{top}(\mathcal{H}, \mathcal{H}_{1}).
\end{equation*}

\end{Prop}
\begin{proof}
    We will show that $\varphi$ sends a separated subset of $(\mathcal{G},            \mathcal{G}_{1})$ in a separated subset of $(\mathcal{H},            \mathcal{H}_{1})$, and vice versa.

    Given $\eps>0$, we will show that there exists $\delta>0$ such that, if $E  \subset X$ is a $(n, \eps)-$separated subset by $(\mathcal{G},            \mathcal{G}_{1})$ then the subset $\varphi E:= \{        \varphi(x); \ x \in E\}$   is $(n, \delta)-$separated by $(\mathcal{H},            \mathcal{H}_{1})$.     

    Indeed, by the uniform continuity of $\varphi^{-1}$, for              $\eps    >0$ there exists $\delta> 0 $ such                          that for every $x',y'\in Y$ we have
    \begin{equation*}
        d(\varphi^{-1}(x'), \varphi^{-1}(y')) \geq \eps \Rightarrow d(x', y')  \geq \delta.
    \end{equation*}
    Let, $\varphi(x), \varphi(y) \in \varphi E$ be two distinct points. Since $x, y \in E$ are distinct, also,  there exists $g \in \mathcal{G}_{n}^{x} \cap \mathcal{G}_{n}^{y} $ such that $d(g(x), g(y)) \geq \eps$. 

    But, by Corollary \ref{Z},  there exists $h \in \mathcal{H}_{n}^{\varphi(x)}  \cap \mathcal{H}_{n}^{\varphi(y)}$  such that     $g = \varphi^{-1} \circ h \circ \varphi$. Hence, 
\begin{equation*}
    d(\varphi^{-1} \circ h \circ \varphi(x), \varphi^{-1} \circ h \circ \varphi(y))  \geq \eps \Rightarrow d( h \circ \varphi(x),  h \circ \varphi(y)) \geq \delta.
\end{equation*}
Therefore, $\varphi(x)$ and $\varphi(y)$ are $(n,          \delta)-$separated by $(\mathcal{H},            \mathcal{H}_{1})$. Also, since $\varphi$ is a bijection we have $\# E = \# \varphi E$. And, it follows that $s((\mathcal{G}, \mathcal{G}_{1}), n, \eps)  \leq s((\mathcal{H}, \mathcal{H}_{1}), n, \delta)$. So,
\begin{equation}\label{'}
    \limsup           \frac{1}{n} \log s((\mathcal{G},            \mathcal{G}_{1}), n, \eps)  \leq  \limsup           \frac{1}{n} \log s((\mathcal{H},            \mathcal{H}_{1}), n, \delta). 
\end{equation}
The other inequality follows with the exact same arguments using $\varphi^{-1}$. So, for every $\eps>0$ there exists $\delta'>0$, such that 
\begin{equation}\label{*}
    \limsup           \frac{1}{n} \log s((\mathcal{G},            \mathcal{G}_{1}), n, \delta')  \geq  \limsup           \frac{1}{n} \log s((\mathcal{H},            \mathcal{H}_{1}), n, \eps).
\end{equation}
Then, the proposition follows taking limits.
\end{proof}

\section{Homogeneous Measures and Expansiveness}
\label{s.arbmor}

In this section, we prove among other results  Theorem B, which is a criterion for the expansiveness of homogeneous ergodic and invariant measures. 

The next definition was presented in \cite{AB}, and shows a class of measures for what the local upper measure entropy assumes a single value for every point, and the same happens for the local   lower measure entropy.

\begin{Def}\cite{AB}
    Let  $(\mathcal{G},   \mathcal{G}_{1}) 
    $ be a finitely generated pseudogroup.  A Borelian           measure   $\mu$  
        is said              to be   $\mathcal{G}-$homogeneous if 

    \begin{enumerate}
        \item $\mu(K) < \infty,$  for every compact subset $ K \subset X$ ;

        \item There exist  a compact $K_{0} \subset K$  such that $ \mu(K_{0})>0$;

\item For     every  $\varepsilon > 0$, there exist $ \delta>0$  and  $c>0$  such that
\begin{equation*}
    \mu(B_{n}^{\mathcal{G}}(y, \delta)) \leq c\mu(B_{n}^{\mathcal{G}}(x, \eps))
    \end{equation*}
    \end{enumerate}
    for every $ n \in \mathbb{N}$ and $x, y \in X$.
\end{Def}

The following result on the entropy of homogeneous measures was obtained in \cite{AB}. Actually, the same proof can be applied, since our dynamical balls are slightly different from the ones in \cite{AB}.

\begin{Lema}[Lemma 4.10 of \cite{AB}]
    If  $\mu$ is  $\mathcal{G}-$homogeneous  in X,                 then 
$$       \overline{h}_{\mu}((\mathcal{G}, \mathcal{G}_{1}),x)  =  \overline{h}_{\mu}((\mathcal{G}, \mathcal{G}_{1}),y)
\textrm{    and }
         \underline{h}_{\mu}((\mathcal{G}, \mathcal{G}_{1}),x)  =  \underline{h}_{\mu}((\mathcal{G}, \mathcal{G}_{1}),y),$$
    for every $x,  y \in X$.
    \end{Lema}



In the next result we present Theorem B as stated in the introduction.

\begin{Teo}[Criterion for Weakly Expansive Measures]\label{89}
       Every ergodic invariant and  $(\mathcal{G}, \mathcal{G}_{1})  -$homogeneous measure of a pseudogroup $( \mathcal{G},  \mathcal{G}_{1})$ with  positive  local upper measure entropy is $(\mathcal{G}, \mathcal{G}_{1})$-weakly expansive.
\end{Teo}

\begin{proof}
    Let $\mu$     an ergodic measure   $\mathcal{G}-$homogeneous by the  pseudogrup $( \mathcal{G},  \mathcal{G}_{1})$.   Let $\delta > 0$ and define
  \begin{equation*}
        X_{\delta}  = \{ x \in X; \ \mu(\Phi_{\delta}(x)) = 0\}.
    \end{equation*}
   We have to show that there exists  $\delta > 0$ such                        that $\mu(X_{\delta}) = 1 $. To do this, 
we will  prove that $X_{\delta}$ is $\mathcal{G}$-invariant.   

By Proposition \ref{.}, it is enough to show that
$$g_{i}^{ -1} (X_{\delta} \cap D_{g_{i}^{-1}})  \subset X_{\delta}, \ \forall g_{i} \in  \mathcal{G}_{1}.$$
Recall that the generating set is symmetric. Let $\delta_{0} > 0$  be the     Lebesgue number of the open covering $\{D_{g}; \ g \in \mathcal{G}_{1}\}$ of $X$ and take $0<\delta < \delta_{0}$.

\textbf{Claim:} $   \Phi_{\delta}(x)          \subset         \bigcup_{g_{i} \in \mathcal{G}_{1}} g_{i}^{ -1} (                 \Phi_{\delta}(g_{i}(x)) \cap D_{g_{i}^{-1}})                        $.

To prove the Claim, let $z \in                \Phi_{\delta}(x) $. Then, 
\begin{equation}\label{78}
     d(h(x), h(z)) \leq \delta, \              \forall h \in      \mathcal{G}_{n}^{x} \cap  \mathcal{G}_{n}^{z},\                               \forall n \in \mathbb{N};
\end{equation}
In particular for $h =  id_{X}$, we obtain 
\begin{equation*}
    d(x, z)         \leq \delta.
\end{equation*}
Hence, using the Lebesgue number, there exists $g_{i}  \in  \mathcal{G}_{1}$      such that   $x, z \in D_{g_{i}}$.

Now, let   $f \in \mathcal{G}_{n}^{g_{i}(x)} \cap  \mathcal{G}_{n}^{g_{i}(z)}$,                          for some $n \in \mathbb{N}$ and take $h = f \circ g_{i}$. Note that $h \in  \mathcal{G}_{n+1}^{x} \cap  \mathcal{G}_{n+1}^{z}$. Hence by (\ref{78}) we have 
\begin{equation*}
    d(f(g_{i}(x)), f(g_{i}(z))) \leq \delta,
\end{equation*}
therefore, $z \in g_{i}^{-1}(\Phi_{\delta}(g_{i}(x))  \cap D_{g_{i}^{-1}})$, concluding the proof of the Claim.

So, taking $x \in g_{i}^{-1}(X_{\delta}  \cap D_{g_{i}^{-1}})$,               we have by $\mathcal{G}-$invariance of $\mu$ 
\begin{center}
    $
    \begin{array}{ccl}
         g_{i}(x) \in X_{\delta} &  \Longrightarrow & \mu(\Phi_{\delta}(g_{i}(x))) = 0\\
         \\
          &  \Longrightarrow & \mu(\Phi_{\delta}(g_{i}(x)) \cap D_{g_{i}^{-1}}) = 0\\
          \\
           &  \Longrightarrow & \mu(g_{i}^{-1}(\Phi_{\delta}(g_{i}(x)) \cap D_{g_{i}^{-1}})) = 0\\
           \\
            &  \Longrightarrow & \mu(\Phi_{\delta}(x)) = 0\\
    \end{array}
    $
\end{center}
therefore $x                      \in X_{\delta}$, showing that $X_{\delta}$ is $\mathcal{G}-$invariant.

Then,                                     since $\mu$ is ergodic, it follows that  $\mu(X_{\delta}  ) = 0 $ or 1.
Consider the function 
\begin{equation*}
\phi_{\delta}(x)  := \limsup_{n \to  \infty}    \frac{-1}{n} \log     \mu(B_{n}^{\mathcal{G}}[x, \delta]). 
\end{equation*}
For every $m \in \mathbb{N}^{+}$, let 
\begin{equation*}
          X^{m} = \{x \in X; \ \phi_{\frac{1}{m}}(x) > \frac{\overline{h}_{\mu}(\mathcal{G}, \mathcal{G}_{1}) }{2} =: h \}.
\end{equation*}
Note that, if $\delta < \delta'$ then $\phi_{\delta'}(x)  <    \phi_{\delta}(x)$. Hence,
\begin{equation*}
    A:= \{ x \in   X;  \  \sup_{\delta > 0}      \phi_{\delta}(x) = \overline{h}_{\mu}(\mathcal{G}, \mathcal{G}_{1}) \} \subset \bigcup_m X^{m}. 
\end{equation*}
  Since the measure is homogeneous, we have  $A = X$. Therefore, there exists an integer $m_0$ such that $\mu(X^{m_0}) > 0$.  Then, 
\begin{center}
$
\begin{array}{ccl}
x \in X^{m_0} & \Rightarrow & \phi_{\frac{1}{m_0}}(x) > h \\
\\
& \Rightarrow  & \displaystyle\limsup_{n \to  \infty}    \frac{-1}{n} \log     \mu(B_{n}^{\mathcal{G}}[x, \frac{1}{m_0}])   > h   \\
\\
& \Rightarrow  & \exists \  n_{k} \ \text{subsequence such that}  \  \frac{-1}{n_{k}} \log     \mu(B_{n_{k}}^{\mathcal{G}}[x, \frac{1}{m_0}])   > h     \\
\\
&  \Rightarrow  & \log     \mu(B_{n_{k}}^{\mathcal{G}}[x, \frac{1}{m_0}])    < -hn_{k}  \\
\\
&  \Rightarrow  &  \mu(B_{n_{k}}^{\mathcal{G}}[x, \frac{1}{m_0}])    < e^{-hn_{k}}\\
\\
& \Rightarrow  & \displaystyle\lim_{k  \to \infty } \mu(B_{n_{k}}^{\mathcal{G}}[x, \frac{1}{m_0}])  = 0.
\end{array}
$
\end{center}

\vspace{0.3cm}

Take                     $ 0<\delta < \min\{\frac{1}{m_0}, \delta_{0}\}$. Hence, since $B_{n_{k}}^{\mathcal{G}}[x, \delta]   \subset   B_{n_{k}}^{\mathcal{G}}[x, \frac{1}{m_0}]$ we also have $\displaystyle\lim_{k  \to \infty } \mu(B_{n_{k}}^{\mathcal{G}}[x,    \delta])  = 0$. Notice that, the Bowen ball is written as
    \begin{equation*}
     \Phi_{\delta}(x) =  \bigcap_{n=1}^{\infty}  B_{n}^{\mathcal{G}}[x, \delta]
     \end{equation*} 
    and it is a nested intersection, because $ n \le s$ implies $B_{s}^{\mathcal{G}}[x, \delta]    \subset B_{n}^{\mathcal{G}}[x, \delta]$. It      follows that 
    \begin{equation*}
    \mu(\Phi_{\delta}   (x)) =  \mu( \bigcap_{n       =1}^{\infty} B_{n}^{\mathcal{G}}[x, \delta]) \leq  \mu( \bigcap_{k       =1}^{\infty} B_{n_{k}}^{\mathcal{G}}[x, \delta]) = \lim_{k  \to \infty } \mu(B_{n_{k}}^{\mathcal{G}}[x, \delta])  = 0,
    \end{equation*}
     therefore  $ x \in  X_{\delta}$. We conclude that $X^{m_0}\subset X_{\delta}$, and since $X^{m_0}$ has positive measure, we have $  \mu( X_{\delta}) =1$. In other words, $\mu$ is  $(\mathcal{G},\mathcal{G}_{1})$-weakly expansive.
    \end{proof}

Due, to our previous discussion on the relations of expansive and weakly expansive measures, we obtain the following criterion.

    \begin{Cor}[Criterion for       Expansive Measures]
        Let $(\mathcal{G}, \mathcal{G}_{1})$ be a good pseudogroup and suppose that the compacted generating set $\mathcal{G}_{2}$ is symmetric. Then, 
         every ergodic and  $(\mathcal{G}, \mathcal{G}_{2})-$homogeneous measure of the pseudogroup $( \mathcal{G},  \mathcal{G}_{2})$ with  positive local upper measure  entropy is $(\mathcal{G},\mathcal{G}_{1})$-expansive.
    \end{Cor}

\begin{proof}
    By   Theorem \ref{89}, we have that $\mu$ is $(\mathcal{G}, \mathcal{G}_{2})-$weakly expansive. And, by Theorem   \ref{qtp}, it  follows that      $\mu$ is $(\mathcal{G}, \mathcal{G}_{1})-$expansive. 
\end{proof}

\begin{mar}\label{remark de grupo}
     
     By  Brin-Katok's Entropy Formula for amenable group actions  (Theorem 2.1 of \cite{BrinKfolner}), the variational principle \cite{Kerr} and the ergodic decomposition theorem \cite{citar}, we           obtain not only a different version  of the Theorem of Criterion of Expansive measures for group actions, dropping the homogeneous hypothesis of the measure, but also the existence of expansive measures, as follows.
\end{mar}

\begin{Cor}\label{cor de gr}
    Let $X$ be a compact metric space and             $G    $ a discrete countable amenable group of homeomorphisms in $X$. If there exists an increasing tempered Folner sequence $\{F_{n}\}$ in $G$                         satisfying 
    \begin{equation*}
        \lim_{n \rightarrow \infty} \frac{|F_{n}|}{\log n }     = \infty
    \end{equation*}
    then every $G-$ergodic Borel probability with positive entropy of $G$ is expansive. In particular, any discrete countable amenable group of homeomorphisms with positive topological entropy has an expansive measure.
\end{Cor}

\begin{Ex}
    Let                                $\Sigma = \{0,  1\}^{\mathbb{Z}} = \{\overline{x} =(\ldots, x_{0}, \ldots); \ x_{i} \in \{0,    1\}, i \in \mathbb{Z}\}$ be the shift space,            $\sigma: \Sigma \longrightarrow \Sigma$ be the shift   map given by $\sigma((x_{i})_{i \in \mathbb{Z}}) = (x_{i+1})_{i \in \mathbb{Z}}$ and consider the                   metric $d(\overline{x},   \overline{y})  = \sum_{k \in \mathbb{Z}} \frac{|x_{k}  - y_{k}|}{2^{k}}$ in    $\Sigma$.

    Define $\mathcal{G}$ the pseudogroup in       $\Sigma$  generated by the set $\mathcal{G}_{1} = \{\sigma, id_{\Sigma}, \sigma^{-1}\}$ and take the Bernoulli measure $\mu_{p}$ in $\Sigma$                  given by the probability vector  $p=(\frac{1}{2}, \frac{1}{2} )$.    We have the following properties:
\\
    \begin{enumerate}[label=(\arabic*)]
\item $\mu_{p}$            is $\mathcal{ G}-$invariant.

Indeed, let $g \in \mathcal{G}$ and $A  \subset D_{g}$ a Borel subset. 

Then,     by definition for each $x\in A$ there exist $g_{1}, \ldots, g_{n} \in \mathcal{G}_{1}$ and a neighborhood $U_{x}$   of $x$ such that 
\begin{equation*}
    g|_{U_{x}}   =               g_{1}  \circ \ldots \circ  g_{n}  |_{U_{x}}.
\end{equation*}
Hence, $\{U_{x}; \ x \in A\}$ is an open cover of                                               $A$ and since $\Sigma$ is a separable metric space, there exists $\{U_{x_{i}}; \ i\in \mathbb{N}\}$ countable subcover of $A$.
Define the sets $V_{1} = U_{x_{1}}\cap A$,  $V_{2} = (U_{x_{2}} \cap A)  \backslash V_{1}$, $\ldots$, $V_{n} = (U_{x_{n}} \cap A)\backslash V_{n-1}$, $\ldots$ and note         that these sets are pairwise disjoints. Also,
$A = \bigcup_{i=1}^{\infty} V_{i}.$
Since $g|_{V_{i}}$  is a composition of          elements of $\mathcal{G}_{1}$ follows that the only ways it can be written are like
\begin{equation*}
    g|_{V_{i}} = \sigma^{l}, \ \text{for some}\ l \in \mathbb{Z},
\end{equation*}
and since $\mu_{p}$ is invariant by each single $\sigma^{l}$ follows that 
    \begin{center}
        $
        \begin{array}{ccl}
            \mu_{p}(g(A)) & = & \mu_{p}(g(\bigcup_{i} V_{i}))
              =  \mu(\bigcup_{i}g(V_{i}))
            
            = \sum_{i} \mu (g|_{V_{i}}(V_{i}))\\
            \\
            & =    & \sum_{i} \mu(V_{i})
            
             =           \mu(\bigcup_{i} V_{i})
            
             =  \mu(A).
        \end{array}
        $
    \end{center}
Therefore, $\mu_{p}$  is $\mathcal{G}-$invariant.
\\
     \item $\mu_{p}$ is $\mathcal{G}-$ergodic.

     Indeed, if $A \subset \Sigma$ is $\mathcal{G}-$invariant then, in particular,        is $\mathcal{G}_{1}-$invariant, hence since $\mu_{p}$ is $\mathcal{G}_{1}-$ergodic follows that 
    \begin{equation*}
        \mu_{p}(A) = 0  \  \text{or}    \ 1.
    \end{equation*}

    \item $\mu_{p}$     is $\mathcal{G}-$homogeneous.

    Indeed, first note that $\mu_{p}$ trivially satisfies the conditions                                $(i)$ and $(ii)$ of the definition of homogeneous measure. 
Now, given $\eps  > 0$, choose $s = \min\{m \in \mathbb{N}; \ \frac{1}{2^{m}}  < \eps   \}, $ then
\begin{center}
    $
    \begin{array}{ccl}
    B_{n}^{\mathcal{G}}(x, \eps) & =&  \{ \overline{z} \in \Sigma;     \ d(g(\overline{x}), g(\overline{z})) \leq  \eps, \ \forall g \in \mathcal{G}_{n}^{\overline{x}} \cap  \mathcal{G}_{n}^{\overline{z}} \}\\
    \\
    & = & \{\overline{z} \in \Sigma; \ d(\sigma^{k}( \overline{x}), \sigma^{k}  (\overline{z}))  \leq \eps; \ -n \leq k \leq n \}\\
    \\
    & = & \{\overline{z} \in \Sigma; \ z_{i} = x_{i}, \ -(n+ s)  \leq i \leq      s+n   \}\\
    \\
     &   = & [x_{-(n+s)}, \ldots, x_{n+s}].
\end{array}
    $
\end{center}
  Taking $\delta = \eps$ and $c = 1$ follows  that 
\begin{equation*}
    \mu_{p}(B_{n}^{\mathcal{G}}(y, \delta)) \leq \mu_p(B_{n}^{\mathcal{G}}(x, \eps)).
\end{equation*}

\item $\overline{h}_{\mu}(\mathcal{G}, \mathcal{G}_{1}) > 0$.

Indeed, we have 
\begin{equation*}
    \overline{h}_{\mu}(\mathcal{G}, \mathcal{G}_{1}) =    \lim_{\eps \rightarrow 0} \limsup_{n} -\frac{1}{n} \log \mu(B_{n}^{\mathcal{G}}(x, \eps)), 
\end{equation*}
 since $\mu(B_{n}^{\mathcal{G}}(x, \eps)) = \log (\frac{1}{2})^{(2(n+s) +1)}$ and
\begin{equation*}
-\frac{1}{n}\log \left( \frac{1}{2} \right)^{(2(n+s) +1)}  = -\frac{(2(n+s)  + 1)}{n}(-\log(2)),
\end{equation*}
follows that 
$\overline{h}_{\mu}(\mathcal{G}, \mathcal{G}_{1}) = 2 \log2$.
      \end{enumerate}

\end{Ex}

\textbf{Question E:} Let $(\mathcal{G},   \mathcal{G}_1)$  be a good pseudogroup. If a  measure $\mu$ is $(\mathcal{G},   \mathcal{G}_{1})-$homogeneous then is it  also  $(\mathcal{G},   \mathcal{G}_{2})-$homogeneous? Or if a  measure $\mu$ is $(\mathcal{G},   \mathcal{G}_{2})-$homogeneous then is it  also  $(\mathcal{G},   \mathcal{G}_{1})-$homogeneous?

\section{Equicontinuity}
\label{s.equi}

In this section we will study   the concept  of uniform equicontinuity for a pseudogroup, as presented in \cite{Ac}. We will prove that every pseudogroup generated  by an finitely generated and uniformly equicontinuous group has no expansive measures.
For what follows, consider $X$ a compact metric space.

\begin{Def}\cite{Ac}
    Let      $\mathcal{G}$ be a    pseudogroup. A subset $G \subset \mathcal{G}$ is said to satisfy the uniform equicontinuity condition if for every                             $\eps > 0$, there exists $\delta > 0$, such that for all $x, y \in X$ and for every $g \in G$ with $x, y \in D_{g}$ we have 
\begin{equation*}
    d(x,y) < \delta \Longrightarrow d(gx, gy) < \eps.
\end{equation*}
\end{Def}

\begin{Def}\cite{Ac}
    A pseudogroup $\mathcal{G}$ is said to be equicontinuous if it has a generating set that is closed under the operations of composition and inversion and satisfies the uniform equicontinuity condition.
\end{Def}

The next result is Theorem C in the Introduction and states that there is a class of pseudogroups that has no expansive measures.

\begin{Teo}\label{5}
   Let $G$ an uniformly equicontinuous group of homeomorphisms in $X$ finitely generated by $G_{1}$. Then, the pseudogroup $\mathcal{G}(G)$ generated by $G$ has no $(\mathcal{G}(G), G_{1})$-weakly expansive measures.
\end{Teo}

\begin{proof}

   Firstly note that $\mathcal{G}(G)  = \mathcal{G}(G_{1})$.

 Suppose by   contradiction, that $\mathcal{G}(G)$    has an  $(\mathcal{G}(G), G_{1})$-weakly expansive measure $\mu$,    i.e., there exists $\rho > 0$ such that 
$\mu(\Phi_{\rho}(x)) = 0$  
 a.e. $x \in X$. Since $G$ satisfies  the uniform equicontinuity condition, we have that  for $\epsilon = \rho$ there exists $\delta > 0$  such that  for every $x, y \in X$ and for every  $ g      \in G$ with $x, y \in D_{g}=X$ we have
\begin{equation*}
d(x,y) < \delta \Longrightarrow d(g(x), g(y)) < \rho.
\end{equation*}

\textbf{Claim:} $ B(x, \delta)  \subset \Phi_{\rho}(x)$.

If not, there is $y \in                   B(x, \delta)$ such that $ y \notin \Phi_{\rho}(x)$. Thus, there is                 $g \in \mathcal{G}_{n}^{x} \cap \mathcal{G}_{n}^{y}$ and $n \in \mathbb{N}$ such that $d(x,y)<\delta$ and $d(gx, gy) > \rho$. But since such a $g$ also belongs to $G$, this contradicts the fact that $\mathcal{G}$ is uniform                  equicontinuous. 


Let $X_{\rho} = \{ x \in X; \mu(\Phi_{\rho}(x)) = 0)\}$ and take $   \mathcal{I} = \{B(x, \delta), \ x \in X_{\rho}\}$ be a covering of open balls for $\overline{X_{\delta}}$. Hence, by compactness,  $\overline{X_{\rho}}$  can be covered by a finite subcovering $\{  B(x_{i}, \delta); \ x_{i} \in X_{\rho}, \ i=1  , \ldots, n\}$, then 
\begin{equation*}
\mu(\overline{X_{\rho}})  \leq    \sum_{i=1}^{n} \mu(B(x_{i}, \delta))               \leq \sum_{i=1}^{n} \mu(\Phi_{\rho}(x_{i}))  = 0,
\end{equation*}
which is a contradiction, since $\mu(\overline{X_{\rho}}) = 1$.
So,                                  
$\mathcal{G}(G)$ has no $(\mathcal{G}(G), G_{1})$-weakly expansive measure.
\end{proof}

Since every expansive measure is weakly expansive, follows from the previous Theorem that $(\mathcal{G}(G), G_{1})$ has no expansive measure, as is stated in the next Corollary.

\begin{Cor}\label{equic}
    Let $G$ an uniformly equicontinuous group of homeomorphisms in $X$ finitely generated by $G_{1}$. Then, the pseudogroup $\mathcal{G}(G)$ generated by $G$ has no $(\mathcal{G}(G), G_{1})$-expansive measures.
\end{Cor}

\begin{Ex}
Let $G_{1}$ be a finite set of isometries defined in  $X$ and consider $G$ the group generated by $G_{1}$. Then,
    the pseudogroup $\mathcal{G}(G)$  generated by $G$   has no expansive measure.

    Indeed, 
    for         $\eps > 0$ take $\delta = \epsilon$, hence for every $x, y                        \in X$ and every $g \in           G $ we have 
    \begin{equation*}
        d(x, y ) < \delta \Longrightarrow d(gx, gy) = d(x, y) < \eps.
    \end{equation*}
    Therefore, $G$ satisfies the uniform equicontinuity condition, then by Corollary  \ref{equic}, has no $(\mathcal{G}(G), G_{1})-$ expansive measures.
\end{Ex}

\begin{Lema}\label{le}
    Let $(\mathcal{G}, \mathcal{G}_{1})$ be a good pseudogroup and $\mathcal{G}_{2}$ the compacted generating set. Suppose $\mathcal{G}$ is equicontinuous with a generating set $\Gamma$. Then, for every $g \in \mathcal{G}_{n}^{2}$ there exist $\lambda > 0$ and $g' \in \Gamma$ such that 
    \begin{equation*}
        x,y \in D_{g}, \ d(x,y)< \lambda \Longrightarrow        g(x) = g'(x) \ \text{and} \ g(y)   =g'(y).
    \end{equation*}
\end{Lema}

\begin{proof}
    Let $g  \in \mathcal{G}_{n}^{2}$ and $x \in K_{g}$. Since $\Gamma$ is also a generating set, there exist $g' \in \Gamma$ and a neighborhood $U_{g}(x)$ of $x$ such that $g|_{U_{g}(x)}  = g'|_{U_{g}(x)}$. 

    Since $K_{g}$ is compact, there exists a finite open covering $\{ U_{g}(x_{1}), \ldots, U_{g}(x_{r})  \}$ of     $K_{g}$. Let $\lambda_{g}$ be the Lebesgue number of this covering, and take     $\lambda = \min \{ \lambda_{g}; \ g \in \mathcal{G}_{n}^{2}         \}$. Hence, if $ d(x,y)< \lambda$ then $g(x) =      g'(x)$ and $g(y)  = g'(y)$.
\end{proof}

\begin{Teo}
   Let $(\mathcal{G}, \mathcal{G}_{1})$  a good and equicontinuous pseudogroup and $\mathcal{G}_{2}$ the compacted generating set. Then, $\mathcal{G}$ has no $(\mathcal{G}, \mathcal{G}_{2})$-weakly expansive measures.
\end{Teo}

\begin{proof}

 Suppose by   contradiction, that $\mathcal{G}$    has an  $(\mathcal{G}, \mathcal{G}_{2})$-weakly expansive measure $\mu$,    i.e., there exists $\rho > 0$ such that 
$\mu(\Phi_{\rho}^{2}(x)) = 0$  
 a.e. $x \in X$. Since $\mathcal{G}$ is equicontinuous, there exists $\Gamma$ a generating set, closed by operations of composition and inverse and that satisfies the equicontinuity condition. Then, for $\epsilon = \rho$ there exists $\delta > 0$  such that  for every $x, y \in X$ and for every  $ g'      \in \Gamma$ with $x, y \in D_{g'}$ we have 
\begin{equation*}
d(x,y) < \delta \Longrightarrow d(g'(x), g'(y)) < \rho.
\end{equation*}
Let $\lambda > 0$ as in  Lemma \ref{le} and let $\xi = \min\{\delta, \lambda \}$.

\textbf{Claim:} $B(x, \xi)  \subset \Phi_{\rho}^{2}(x)$.

If not, there is $y \in                   B(x, \delta)$ such that $ y \notin \Phi_{\rho}(x)$. Thus, there is    $n \in \mathbb{N}$ and             $g \in \mathcal{G}_{n}^{2,x} \cap \mathcal{G}_{n}^{2,y}$  such that $d(x,y)<\xi $ and $d(gx, gy) > \rho$. 
But since       $\xi < \lambda $   follows by Lemma \ref{le}   that for such a $g$, there exists $g' \in \Gamma$ such that $g(x) = g'(x)$ and $g(y)  = g'(y)$. Hence $d(g'(x),   g'(y)) > \rho$, but this contradicts the fact that $\Gamma$ satisfies the condition of uniform       equicontinuous. 


Let $X_{\rho} = \{ x \in X; \mu(\Phi_{\rho}^{2}(x)) = 0)\}$ and take $   \mathcal{I} = \{B(x, \delta), \ x \in X_{\rho}\}$ be a covering of open balls for $\overline{X_{\rho}}$. Hence, by compactness,  $\overline{X_{\rho}}$  can be covered by a finite subcovering $\{  B(x_{i}, \delta); \ x_{i} \in X_{\rho}, \ i=1  , \ldots, n\}$, then 
\begin{equation*}
\mu(\overline{X_{\rho}})  \leq    \sum_{i=1}^{n} \mu(B(x_{i}, \delta))               \leq \sum_{i=1}^{n} \mu(\Phi_{\rho}(x_{i}))  = 0,
\end{equation*}
which is a contradiction, since $\mu(\overline{X_{\rho}}) = 1$.
So,                                  
$\mathcal{G}(G)$ has no $(\mathcal{G}(G), G_{1})$-weakly expansive measure.
\end{proof}

Similarly as explained previously in this section, since every expansive measure is also weakly expansive, we have the following corollary.

\begin{Cor}
   Let $(\mathcal{G}, \mathcal{G}_{1})$  a good and equicontinuous pseudogroup and $\mathcal{G}_{2}$ the compacted generating set. Then, $\mathcal{G}$ has no $(\mathcal{G}, \mathcal{G}_{2})$-expansive measures.
\end{Cor}

\section{Acknowledgements}

Alexander Arbieto was partially supported by CAPES– Finance Code 001, CNPq Grant 307877/2025-6, PRONEX-Dynamical Systems and FAPERJ “Programa Cientista do Nosso Estado", E-26/201.181/2022 and E-26/200.281/2026.

Luana Segantim was supported by CAPES and CNPq.

Jaqueline Siqueira was supported by
the grant E-26/010/002610/2019, Rio de Janeiro Research Foundation (FAPERJ), and by the Coordena\c c\~ao de Aperfeiçoamento de Pessoal de N\'ivel Superior - Brasil (CAPES), Finance Code 001.

\end{document}